\documentclass[a4paper, 10pt]{article}
\usepackage{amsthm}
\usepackage{amsmath,amssymb}
\usepackage{graphicx}

\newtheorem{theorem}{Theorem}[section]
\theoremstyle{plain}
\newtheorem{conjecture}[theorem]{Conjecture}
\newtheorem{lemma}[theorem]{Lemma}

\newtheorem{corollary}[theorem]{Corollary}
\newtheorem{observation}[theorem]{Observation}
 
\theoremstyle{definition}
\newtheorem{definition}[theorem]{Definition}
\newtheorem{example}[theorem]{Example}

\begin{document}


\title{Improved bounds for the shortness coefficient of cyclically 4-edge connected cubic graphs and snarks.}

\author{Klas Markstr\"om\thanks{Department of Mathematics and Mathematical Statistics, 
Ume\aa\  University , SE-901 87  Ume\aa, Sweden} }

\maketitle


\begin{abstract}
	We present a construction which shows that there is an infinite set of cyclically 4-edge connected cubic graphs on $n$ vertices with no cycle 
	longer than $c_4 n$ for $c_4=\frac{12}{13}$, and  at the same time prove that a certain natural family of cubic  graphs cannot be used to lower the shortness 
	coefficient $c_4$ to 0.
	
	The graphs we construct are snarks so we get the same upper bound for the shortness coefficient of snarks, and we prove that the constructed graphs have an oddness growing 
	linearly with the number of vertices.
\end{abstract}

\section{Introduction}
The length of the longest cycle, called the circumference  and here denoted $circ(G)$,  in a regular graph is a property related to a number of well studied problems and conjectures. Give an infinite family of graphs  $\mathcal{F}$, the \emph{shortness coefficient} of $\mathcal{F}$ is defined to be  $c_{\mathcal{F}}=\liminf_{G\in \mathcal{F}} \frac{circ(G)}{|G|}$, and the  \emph{shortness exponent}   $e_{\mathcal{F}}=\liminf_{G\in \mathcal{F}}  \frac{\log(circ(G))}{\log(|G|)}$.

For 2-connected cubic graphs Bondy and Entringer \cite{BE} showed that the circumference can be as low as $\mathcal{O}(\log n)$.  For 3-connected cubic graphs Bondy and Simonovits \cite{BS} proved that the shortness exponent is at most $\frac{\log 8}{\log 9}$, and Jackson \cite{J} proved that it is at least $\log_2{(1+\sqrt 5)}$.  For the general case of $k$-regular $k$-connected graphs, with $k\geq 3$, Jackson and Parsons \cite{JP} proved that the  shortness exponent is strictly between 0 and 1 for all fixed $k$.  In contrast to this stands Tutte's result \cite{tu2} that every planar 4-connected graph is hamiltonian, and the result by Thomas and Yu \cite{TY} which shows that every projective planar 4-connected graph is hamiltonian as well.

Recall that a graph $G$ is \emph{cyclically $k$-edge connected} if the deletion of fewer than $k$ edges from $G$ does not create two components which both contain at least one cycle.  The cubic graphs constructed in \cite{BS} are not cyclically 4-edge connected and Bondy, see \cite{bo}, conjectured that the shortness coefficient for cyclically 4-edge connected cubic graphs is non-zero.  A much stronger conjecture was made by Thomassen  \cite{th} who conjectured that every 4-connected line graph is hamiltonian.  For line graphs of cubic graphs this conjecture is equivalent to saying that every cyclically 4-edge connected cubic graph has a dominating cycle, in the sense that every edge has at least one endpoint on the cycle.  The latter property had been conjectured separately by Ash and Jackson \cite{AJ}, and for the subclass of non-3-edge colourable cyclically 4-edge connected cubic graphs by Fleischner \cite{Fl}.

These conjectures would imply that the shortness coefficient for cyclically 4-edge cubic graphs is at least $\frac{3}{4}$. For planar cubic graphs this is known to hold by a theorem of Tutte \cite{tu2}. We refer the reader to \cite{BRV} for a survey of these conjectures and their relationship to several other, in many cases equivalent, conjectures.

Going in the other direction, i.e.\ bounding the circumference from above,  several authors have proven that different classes of planar cyclically 4-edge connected cubic graphs have shortness coefficient less than 1, e.g. \cite{p1,p2}.  The lowest explicit such bound we are aware of is that from  \cite{Z} which gives a class with shortness 
coefficient $\frac{85}{86}=0.9883\ldots$. The construction given in \cite{Hagg} shows that $c_4\leq \frac{14}{15}=0.933\ldots$, although this is not explicitly stated in these terms there. Letting $c_k$ denote the shortness coefficient for cyclically $k$-edge connected cubic graphs we hence know that $c_3=0$, $c_4\leq \frac{14}{15}$, and the results of \cite{p2} implies that $c_5\leq \frac{971}{972}$.

The purpose of this paper is to improve the upper bound on $c_4$. In the next section we discuss the circumference for  graphs given by a simple construction and in the following section we give two explicit examples which demonstrate that $c_4 \leq \frac{17}{18}=0.9444\ldots$ and $c_4 \leq \frac{24}{26}=0.92307\ldots$ respectively.

The graphs in our example families are also snarks, i.e.\ cyclically 4-edge connected cubic graphs of girth at least 5 which are not 3-edge colourable. One measure of the degree of non-colourability for a cubic graph is its oddness, i.e.\ the minimum number of odd cycles in any 2-factor of the graph.  All snarks on $n\leq 36$ vertices have oddness 2 \cite{BGHM} but it is known that there are families of snarks with an oddness which is linear in the number of vertices $n$. In \cite{Hagg} families of snarks with oddness at least $\frac{n}{15}$ were constructed, for certain sequences of $n$ and in \cite{LMS} this was improved to $\frac{n}{13}$.  Our second construction matches this  bound and we point out a second way of constructing snarks with the same oddness growth.

In the last section we discuss a property related to the dominating cycle conjecture and we conjecture that our bounds can be improved to show that in fact $c_4=0$.

\section{Graph substitutions and long cycles}

The following general type of construction has been used to construct graph families with various desired properties, in particular the construction in \cite{BS} showing that the shortness exponent of 3-edge connected cubic graphs is less than 1 was of this form.
\begin{definition}\label{contr}
	A graph $G$ is said to be a substitution of a graph $H$ into a multigraph $F$ if there is a partition of $V(G)$ into sets $\{V_i\}_{i=1}^{n_2}$, where $n_2=|V(F)|$ such that
	\begin{enumerate}		
		\item  The multigraph obtained by contracting each $V_i$ to a single vertex $v_i$ is isomorphic to $F$.
		
		\item  There is an edge $e\in E(H)$ such that  each induced graph $H_i=G([V_i])$ is isomorphic to the graph obtained by deleting the vertices of $e$ from $H$.
	\end{enumerate}
	
	We denote any substitution of  $H$ into $F$ using a specified edge $e$ of $H$ by $S(H,F,e)$
\end{definition}
Note that given $F$, $H$, and $e$ there are typically many non-isomorphic graphs of the form $S(H,F,e)$, since given neighbouring vertices $v_i$ and $v_j$ we have not specified which vertices in $H_i$ and $H_j$ the corresponding edge connect.

The following lemma is easy to prove
\begin{lemma}\label{lem1}
	If $H$ is cubic and cyclically $k$-edge connected and $F$ is $k$-regular and $k$-edge connected then $S(H,F,e)$ is a cyclically $k$-edge connected cubic graph.
\end{lemma}

Next we note that under the contraction used in point 1 of Definition \ref{contr}  a cycle $C$ in $S(H,F,e)$  will be transformed into an eulerian subgraph $T$ of $F$.  We can also go in the other direction, as we will now prove for the case $k=4$.
\begin{lemma}
 	Let $G$  be given by a substitution as in Lemma \ref{lem1}, with $k=4$,  and $T$ be an eulerian subgraph of $F$. Then there exists a cycle in $G$ such that the contraction of the $V_i$ in $G$ will turn $C$ into an eulerian trail of $T$.
\end{lemma}
\begin{proof}
	If $v_i$ has degree 2 in $T$ we let $u_1$ and $u_2$ denote the two vertices in $V_i$ which are neghbours of vertices outside $V_i$ in $T$. Since $H$ is 4-connected we can find a longest path $P_i$ from $u_1$ to $u_2$ in $H_i$ of length $l_i$.
	
	If $v_i$ has degree 4 in $T$ we let $u_1,\ldots,u_4$ be the vertices in $V_i$ with neighbours outside $V_i$ in $T$.  Given a partition $\{u_i,u_k\},\{u_l,u_m\}$ we 
	can always, by Menger's Theorem, find two disjoint paths $P_{i,1}$ and $P_{i,2}$, of length $l_{i,1}$ and $l_{i,2}$, in $H_i$ such that both have one endpoint in the  first part of the partition 
	and one in the second part.  
	
	 Recall that a transition in an eulerian trail is a pair of consecutive edges in the trail. The observations above mean that at vertex a $v_i$ of $T$ there is always at least 1 and 4  allowed transitions 
	  through $v_i$, depending on the degree, and by Kotzig's  theorem on eulerian trails with forbidden transitions \cite{Ko} there is an 
	 eulerian trail in $T$ which is compatible with these allowed transitions.  In $G$ this trail corresponds to a cycle with length equal to the sum of the lengths of the paths inside the $V_i$s.
	 
\end{proof}
This lemma can be generalized to larger even values of $k$ in a straight forward way, now deleting more than two vertices from $H$.  If $k$ is odd, and at least 5, we can use a theorem of Catlin, see \cite{cat}, together with \cite{Ja} to conclude that $G$ has a spanning eulerian subgraph and then use the same arguments as for a uniform even $k$ on this subgraph.

From these lemmas we see that we can always find a cycle in $S(H,F,e)$ which contains vertices from each $V_i$, and the maximum possible length of such a cycle will primarily depend on the properties of $H$, rather than those of $F$.  If $F$ is hamiltonian we could take $T$ to be a Hamiltonian cycle, and hence only use a single path in each $V_i$.  From \cite{JP} we know that for each $k\geq 3$ there are $k$-regular $k$-edge connected graphs on $n$ vertices with circumference less than $n^{c_k}$, where $c_k<1$, however if $F$ is of this type we will be able to use a general eulerian subgraph $T$ instead of a cycle. As a corollary we get,
\begin{corollary}
	For any fixed $H$ and sequence of  graphs $F_n$ the shortness coefficient for the class of cyclically 4-edge connected substitutions $S(H,F_n,e)$ is strictly greater than 0.
\end{corollary}

\section{A new upper bound for the shortness coefficient of cyclically 4-edge connected cubic graphs.}
We will here give two families of cyclically 4-edge connected cubic graphs which give improved upper bounds on the shortness coefficient $c_k$.
We let $(F_i)_{i=1}^{\infty}$ be a sequence of 4-regular 4-edge connected graphs.

In each of these examples we will need some facts regarding the maximum length of a cycle either including or excluding certain vertices or edges. These facts can be ascertained by hand, however this analysis is lengthy and not particularly enlightening so here we instead state these facts based on computational tests done using Mathematica. This was  done by finding all cycles in these graphs and then testing the stated properties.

\begin{example}
	Let $H^1$ be the graph in Figure \ref{fig1},  $e$ the edge marked in the figure,  and $G^1_n=S(H^1,F_n,e)$.  We will prove that $c_{\{S(H^1,F_n,e)\}} \leq \frac{17}{18}=0.9444\ldots$

	\begin{figure}
	\begin{center}
	    \includegraphics[width=0.5\textwidth]{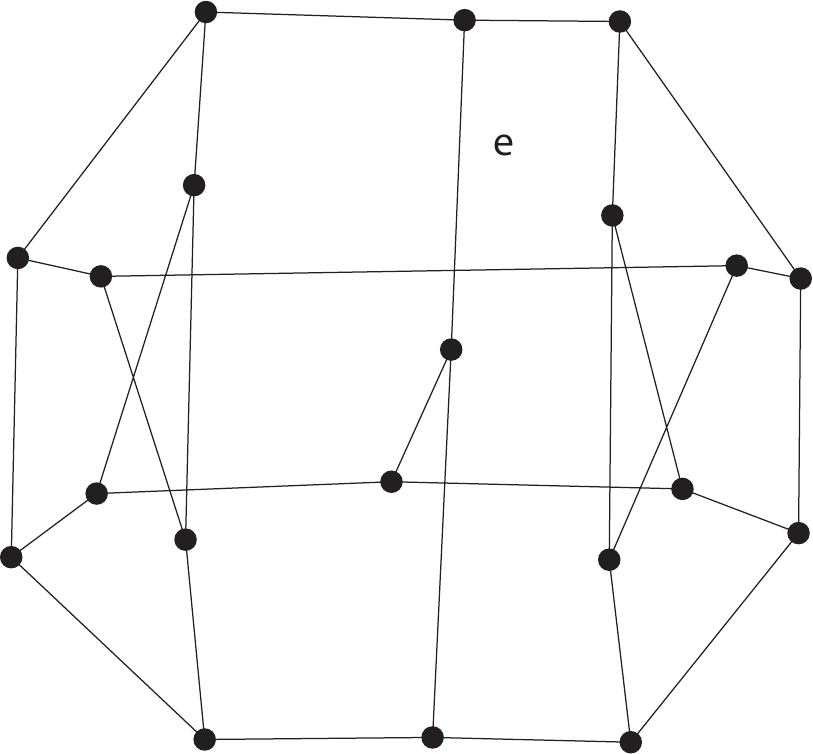}
	    \caption{The graph $H^1$}
    	\end{center}
	\end{figure}\label{fig1}
  
  	Let us consider what the restriction of a cycle in $S(H^1,F_n,e)$ to an $H_i$ can look like.	
	\begin{enumerate}
		\item  
		The first possibility is that the restriction is a single path. This corresponds to two different possibilities in $H^1$
		\begin{enumerate}
			\item  A cycle including the edge $e$.  For $H^1$ this cycle has at most 19 vertices.
		
			\item A cycle using one endpoint of $e$ but not the other. For $H^1$ this cycle has at most 18 vertices.
		\end{enumerate}

		\item  
  		The second possibility is that the restriction is two paths. This corresponds to two different possibilities in $H^1$
		\begin{enumerate}
			\item  A cycle including both endpoints of $e$ but not the edge $e$. For $H^1$ this cycle has at most 19 vertices.
		
			\item Two disjoint cycles each using one endpoint of $e$. For $H^1$ these cycles together have at most 18 vertices.
		\end{enumerate}
	\end{enumerate}
  
	From this we conclude that $G^1_n$ will have $18|F_n|$ vertices and the longest cycle in $G_n$ will have at most $17|F_n|$ vertices, hence giving us the 
	claimed bound for the shortness coefficient.    
\end{example}

\begin{example}
	Let $H^2$ be the graph in Figure \ref{fig2}, let $e$ be the edge marked in the figure, and $G^2_n=S(H^2,F_n,e)$.  We now prove that $c_{\{S(H^2,F_n,e)\}} \leq \frac{24}{26}=0.92307\ldots$

         \begin{figure}
         	\begin{center}
	    \includegraphics[width=0.5\textwidth]{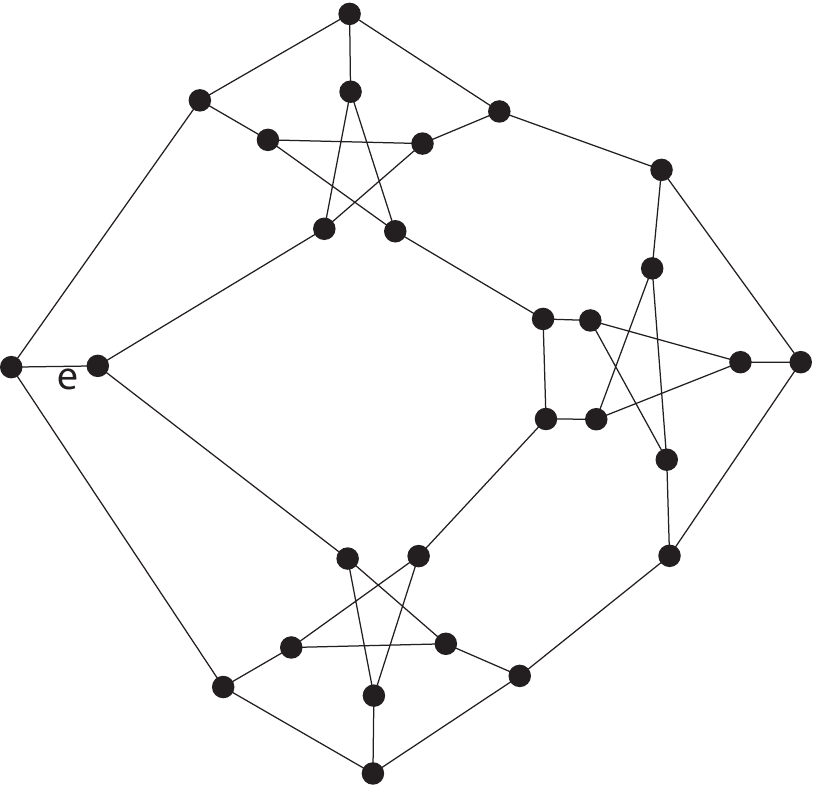}
	    \caption{The graph $H^2$}
        	\end{center}
	\end{figure}\label{fig2}

	Enumerating the way the intersection between a cycle and $H_i$ as before we get the following lengths 
	\begin{enumerate}
		\item  
		\begin{enumerate}
			\item  For $H^2$ this cycle has at most 26 vertices.
		
			\item For $H^2$ this cycle has at most 25 vertices.
		\end{enumerate}
		\item  
		\begin{enumerate}
			\item For $H^2$ this cycle has at most 26 vertices.
		
			\item  For $H^2$ these cycles together have at most 25 vertices.
		\end{enumerate}
	\end{enumerate}

	From this we conclude that $G^2_n$ will have $26|F_n|$ vertices and the longest cycle in $G^2_n$ will have at most $24|F_n|$ vertices, hence giving us the 
	claimed bound for the shortness coefficient.
\end{example}

\subsection{Snarks and the growth of Oddness}
The graph $H^1$ is a snark and every 2-factor of $H^1$ contains an odd cycle which includes no vertex from $e$.  Hence every 2-factor of $G^1_n$ will have 
at least $|F_n|$ odd cycles. This means that $G^1_n$ is also a snark and has oddness at least $|F_n|$. In fact since the oddness is an even number 
the oddness will be at least  $|F_n|+1$ if  $|F_n|$ is odd. Thus the bound for the shortness coefficient applies to the snarks as well. 

The graph $H^2$ is also a snark and every 2-factor of $H^2$ contains two odd cycles which include no vertex from $e$.  Hence every 2-factor of $G^2_n$ will have 
at least $2|F_n|$ odd cycles. This means that $G^2_n$ is also a snark and has oddness at least $2|F_n|$.

Constructing snarks with a high oddness is an interesting problem in itself. See \cite{LMS} for the best current bounds on ratio between the oddness and number of 
vertices for snarks of different cyclic connectivities.   We note that apart from the 28 vertex snark $H^2$ there is  a second snark $H^3$ on 28 vertices, shown in Figure \ref{fig3}, which has an edge 
$e$, marked in the figure, such that every 2-factor has two odd cycles which include no vertex from $e$. Hence this snark can be used in a substitution construction to get a second family of 
snarks with the same oddness ratio as $S(H^2,F_n,e)$. However, for this snark the substitution construction does not yield graphs with small circumference.

         \begin{figure}
                  	\begin{center}
	    \includegraphics[width=0.5\textwidth]{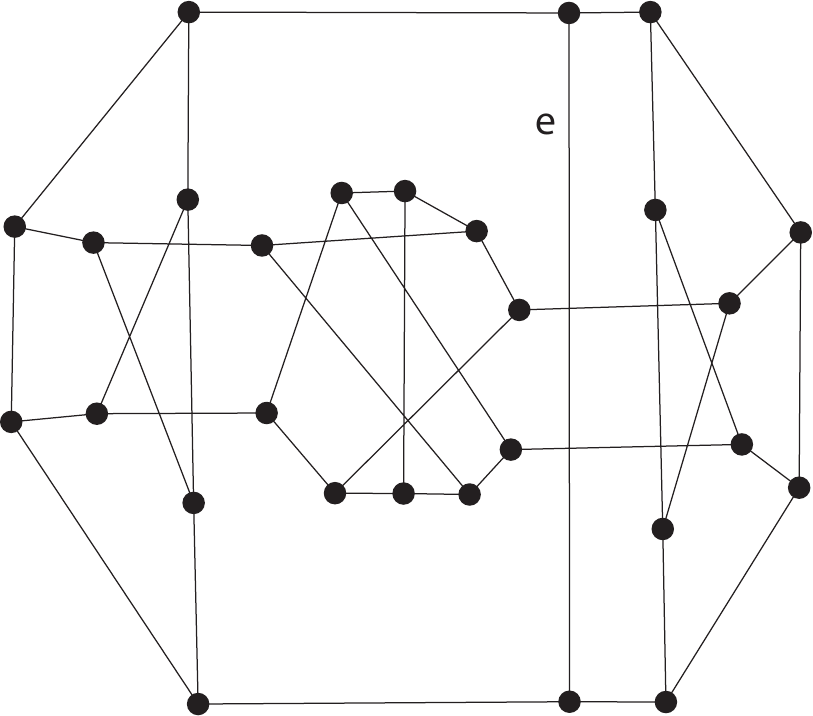}
	    \caption{The graph $H^3$}
	        	\end{center}
	\end{figure}\label{fig3}

The families $G^1_n$, $G^2_n$ and $G^3_n$ all contain a large number of graphs, the growth rates for the number of such graphs on $N$ vertices will be closely related to that of the number of 4-regular graphs on $N/18$ and $N/26$ vertices respectively.  This means that  even though all snarks on $N\leq 36$ vertices have oddness 2, see \cite{BGHM}, this will change drastically for larger $N$, with many snarks of larger oddness and a wide spread of oddness values for each $N$.

\section{A conjecture}
As we have seen the cycle structure of certain cubic graphs becomes more restricted as the number of vertices grows. In both $H^1$ and $H^2$ there are edges which do not belong to any cycle of the same length as the circumference of the graph, and cycles using both endpoints of these edges are also shorter than the circumference.  As both $H^2$ and $H^3$ demonstrate there are also snarks with edges which do not participate in any of the odd cycles of 2-factors with a minimal number of odd cycles. With regard to the dominating cycle conjectures mentioned in the introduction we believe that this more rigid behaviour is cause for concern. 

In \cite{FK} Fleischner and Kochol proved that the dominating cycle conjecture for cyclically 4-edge connected graphs is equivalent to the statement that there is a dominating cycle containing any specified pair of edges from the graph.  Let us consider the following question: Given a matching of size $k$ in $G$, is there a dominating cycle containing the matching? For $k=4$ the answer is negative for the 8 vertex M{\"o}bius ladder, which is cyclically 4-edge connected.   By using Mathematica to find all dominating cycles in the small cyclically 4-edge connected graphs ( the graphs can be obtained from e.g the website of \cite{HOG}) we can make the following observation. Recall that a weak snark is a cyclically 4-edge connected cubic graph which is not 3-edge colourable.
\begin{observation}
	\begin{enumerate}
		\item Given a matching $M$ of size 3 in any cyclically 4-edge connected  cubic graph on $n\leq 22$ vertices there is a  dominating cycle containing $M$. 
		
		\item Given a matching $M$ of size 4 in any weak snark on $n\leq 22$ vertices there is a  dominating cycle containing $M$. This is not true for $n=24$
		
		\item Given a matching $M$ of size 3 in any weak snark on $n\leq 30$ vertices there is a  dominating cycle containing $M$.
			
		\item Given a matching $M$ of size 4 in any snark on $n\leq 28$ vertices there is a  dominating cycle containing $M$. This is not true for $n=30$

		\item Given a matching $M$ of size 3 in any snark on $n\leq 32$ vertices there is a  dominating cycle containing $M$.

	\end{enumerate}
\end{observation}
 We believe that for larger numbers of vertices the dominating cycles in each of the three graph classes will become less flexible, and  eventually fail for $k=2$ as well. As noted before, if the dominating cycle conjecture is true then $c_3\geq \frac{3}{4}$.  Jaeger has made the weaker conjecture, see \cite{FJ}, that every cyclically 4-edge connected cubic graph has a cycle $C$ such that $G\setminus C$ is a forest, which would imply that $c_4\geq \frac{1}{2}$. However, we expect that the constructions given in this paper can be significantly improved upon,
\begin{conjecture}
	$c_4=0$	
\end{conjecture}

\providecommand{\bysame}{\leavevmode\hbox to3em{\hrulefill}\thinspace}
\providecommand{\MR}{\relax\ifhmode\unskip\space\fi MR }
\providecommand{\MRhref}[2]{%
  \href{http://www.ams.org/mathscinet-getitem?mr=#1}{#2}
}
\providecommand{\href}[2]{#2}

\end{document}